\documentclass{article}
\usepackage[margin=1.3 in]{geometry}
\usepackage[english]{babel}
\usepackage{authblk}
\usepackage{amsthm}
\usepackage{amsmath}
\usepackage{amssymb}
\usepackage{amsfonts}
\usepackage[square,numbers]{natbib}
\usepackage{hyperref}
\usepackage{multirow}
\usepackage{babel}
\usepackage{array}
\usepackage[T1]{fontenc}
\newtheorem{theorem}{Theorem}

\theoremstyle{definition}
\newtheorem{exmple}{Example}[section]

\def\ff{\ensuremath{\mathbf{F}}}
\title{Involutary pemutations  over finite fields 
given by trinomials and quadrinomials}
\author{Anitha G and P Vanchinathan}
\begin{document}
\maketitle
\begin{center}
{School of Advanced Sciences\\ 
 Vellore Institute of Technology\\ 
 Chennai 600 127, India\\}
\texttt{anitha.g2019@vitstudent.ac.in, \quad vanchinathan.p@vit.ac.in} 
\end{center}

\begin{abstract}
   For all finite fields of $q$ elements where  $q\equiv1\pmod4$ we
	have constructed permutation polynomials which have order 2 as 
	permutations, and have 3 terms, or 4 terms as polynomials.
	Explicit formulas for their coefficients are given in terms
	of the primitive elements of the field.
	We also give polynomials providing involutions with 
	larger number of terms but coefficients will be conveniently
	only two possible 
	values. Our procedure gives at least $(q-1)/4$ trinomials, and
	$(q-1)/2$  quadrinomials, all 
	yielding involutions with unique fixed points over  a field of 
	order $q$. Equal number of  involutions with 
	exactly $(q+1)/2$ fixed-points are provided as quadrinomials.
\end{abstract}


\section{Introduction}
Permutation polynomials which are  of order 2 as elements of the  permutation
group are called  involutions. Being their own inverses as permutations,
they provide some extra advantage  in  applications to cryptography. 

Earlier work on involutary pemutations have been carried out by \cite{charpin,niu,zheng,zha} among many others.
Another practical requirement is to find  polynomials  with fewer terms as 
they can be evaluated fast. 
In that respect many authors have worked on producing permutation 
binomials, trinomials and quadrinomials.
Recent results of Bartoli and Zini \cite{bz} and the work of Rai and Gupta 
\cite{rai} are negative statements:
they investigate trinomials of some restricted type
and conclude that  most of them are  not permutation polynomials.
Their  results underscore the difficulty
of locating  trinomials providing permutations.
However there has been limited literature focussing on producing 
permutation polynomials of fewer terms as well as producing involutions.
That is producing polynomials  where both the combinatorics (cycle structure)
and the algebra (fewer terms)  can be determined. One significant paper the authors are aware of is \cite{meidl} which determines cycle structure of permutation polynomials of Carlitz rank upto 3.

We report our success on this endeavour. 
\textit{All the findings in this paper are involutory permutations  with a majority of them
over fields of order  $q\equiv1\pmod 4$.}
We are able to construct $(q-1)/4$ permutation trinomials in each of these field (See the Appendix for examples of such polynomials).
In Section 4, we consider finite fields of general order including those in characteristic 2.

Besides trinomials, we also construct  quadrinomials and another
collection of polynomials with $2d$ terms with $d$ a factor of $q-1$, where
$q$ is the order of the finite field. To compensate the 
complexity with higher number of terms they have a convenience of all nonzero coefficents being $a$ or $-a$ for a fixed element $a$.

The polynomials  we construct are of  the form $x^r h(x^s)$ for $s|q-1$.
These are  studied extensively in papers 
dealing with  cyclotomic mappings. See e.g., \cite{hou,Hou,wang2007,wang2013}. 

We describe below the  results we have obtained now.
\vspace{.5pc}
\begin{theorem} \label{t1}
	For a finite field $\ff_q$ with $q\equiv1\pmod 4$,  
     let $d=(q-1)/4$, and $g$ be a generator
	of the multiplicative group. 
	Depending on  a choice of $i, 0\leq i \leq d-1$ define $a_0,a_1$ and $a_2$ as follows:
      \[a_0=\frac{g^{2(4i+2)}+1}{2g^{4i+2}},\quad a_1=\frac{(g^d+1)(g^{2(4i+2)}-1)}{4g^{4i+d+2}},\quad a_2=g^d a_1\]
 Then    $f(x)=a_2x^{3d+1}+a_1x^{d+1}+a_0x $ is an involutiory
	permutation polynomial over $\ff_{q}$ 
	with zero as the only fixed-point.
   \end{theorem}

As  $i$ varies  we see that by this method 
we get   $(q-1)/4$ involutions  
over a single finite field, all represented as trinomials.

   We describe our next result which also provides involutions without
   nonzero fixed points, but now represented by  quadrinomials.

\begin{theorem} \label{t2}
	Under the same notations as in the previous theorem,
	for a choice of $i, 0\leq i \leq d-1$ define $a_0,a_1,a_2$ and $a_3$ as follows:
      \[a_0=\frac{(g^2+1)(g^{2(4i+2)}+1)}{4g^{4i+3}}, \quad
	a_1=\frac{(g^{d+2}+1)(g^{2(4i+2)}-1)}{4g^{4i+d+3}},\]
	\[a_2=\frac{(g^2-1)(g^{2(4i+2)}+1)}{4g^{4i+3}},  \quad
       a_3=\frac{(g^{d+2}-1)(g^{2(4i+2)}-1)}{4g^{4i+d+3}}\]
 Then  the polynomial $f(x)=a_3x^{3d+1}+a_2x^{2d+1}+a_1x^{d+1}+a_0x$ is an involution  over $\ff_{q}$ with zero as the only fixed-point.
   \end{theorem}

Next result provides involutions with more fixed points.
\begin{theorem} \label{t3}
	Under the same notations as above,
for any choice of $i$, $0\leq i \leq d-1 $ we define
	coefficients $a_3,a_2,a_1$ and $a_0$ in one of the
	following two  ways: 
\begin{enumerate}
	\item[(a)] $
		\displaystyle
	a_0=\frac{(g^{4i+2}+1)^2}{4g^{4i+2}}, \quad
	a_1=\frac{g^{2(4i+2)}-1}{4g^{4i+2}},\quad
	a_2=\frac{(g^{4i+2}-1)^2}{4g^{4i+2}}, \quad
	a_3=a_1 $  
\item[(b)]  $ 
	\displaystyle
	a_0=\frac{(g^{4i+2}+1)^2}{4g^{4i+2}},\quad
	\quad a_1=\frac{g^{2(4i+2)}-1}{4g^{d+4i+2}},\quad
	\quad a_2= 1-a_0,\quad a_3=-a_1 $
\end{enumerate}
In both the cases, the  polynomial
\[f(x)=a_3x^{3d+1}+a_2x^{2d+1}+a_1x^{d+1}+a_0x\] 
is an involution over $\ff_q$ with $\frac{q+1}{2}$ fixed-points.
\end{theorem}


This paper is organized as below: In Section 2 we give proofs for
the construction of involutions without non-zero fixed points.
In Section 3, proofs are provided for involutions with more fixed points.
In Section 4, we consider fields of arbitrary order including those of 
characteristic 2.  Theorems 4 and 5 there provide a  pair 
of interesting involutions which have some kind of interesting duality
relationship between them. Two more results give quadrinomial involutions in
	odd characteristic with coefficents from a restricted set.
In the Appendix we give a table of polynomials constructed using our theorems.


\section{Fixed-point-free involutions}
  We now  proceed to prove that  our construction of trinomials,
given in Theorem \ref{t1} are involutions.

\begin{proof}
	For this purpose   we write the  nonzero elements 
  as $g^j,j=0,1,\ldots,q-2$, and look at $j$ modulo 4;
then we will show that $f$ interchanges, for each $k$, the  
	element $g^{4k}$ with $g^{4k+4i+2}$, and 
	the element $ g^{4k+1}$ with $ g^{4k+4i+3}$.

	\noindent Case (i): Evaluating at  $g^j$  for $j=4k$.\\
\begin{flalign*}
   f(g^{4k})&=a_2(g^{4k})^{3d+1}+a_1(g^{4k})^{d+1}+a_0 g^{4k} &\\
    &=g^{4k}\bigl(a_2(g^{4k})^{3d}+a_1(g^{4k})^d+a_0 \bigr)&\\
    &=g^{4k}(a_2+a_1+a_0) &\\
	&=g^{4k}\Bigl( \frac{(g^d+1)(g^{2(4i+2)}-1)}{4g^{4i+d+2}} (g^d+1)+ \frac{g^{2(4i+2)}+1}{2g^{4i+2}} \Bigr)
	\mbox{ (substituting for $a_j$'s) } &\\
 &=g^{4k}\Bigl( \frac{2g^d(g^{2(4i+2)}-1)}{4g^{4i+d+2}}+ \frac{g^{2(4i+2)}+1}{2g^{4i+2}} \Bigr)&\\
 &=g^{4k}\Bigl( \frac{g^{2(4i+2)}-1}{2g^{4i+2}}+ \frac{g^{2(4i+2)}+1}{2g^{4i+2}} \Bigr) &\\
    &=g^{4(k+i)+2}&\\
\end{flalign*}
\noindent Case (ii):  Evaluating at  $g^j$  for $j\equiv 2\pmod 4$.\\                  
\begin{flalign*}
   f(g^{4(k+i)+2})&=a_2(g^{4(k+i)+2})^{3d+1}+a_1(g^{4(k+i)+2})^{d+1}+a_0 g^{4(k+i)+2} &\\
    &=g^{4(k+i)+2}(a_2g^{2d}+a_1g^{2d}+a_0) &\\
	&= g^{4(k+i)+2}(-g^da_1-a_1+a_0) \mbox{ (since } g^{2d}=-1)  \\
   &=g^{4(k+i)+2}\Bigl(\frac{-g^{2(4i+2)}+1}{2g^{4i+2}}
	+\frac{g^{2(4i+2)}+1}{2g^{4i+2}}\Bigr) 
	\mbox{(by substituting the values of } a_0, a_1) &\\
    &=g^{4(k+i)+2}g^{-4i-2} &\\
    &=g^{4k}
\end{flalign*}
\noindent Case(iii): Evaluating at  $g^j$  for $j\equiv 1\pmod 4$.\\
Now $f(x)$ at $g^{4k+1}$
\begin{flalign*}
   f(g^{4k+1})&=a_2(g^{4k+1})^{3d+1}+a_1g^{4k+1})^{d+1}+a_0 g^{4k+1} &\\
    &=g^{4k+1}\bigl(a_2(g^{4k+1})^{3d}+a_1(g^{4k+1})^d+a_0 \bigr)& \\
    &=g^{4k+1}(-a_2g^d+a_1g^d+a_0) &\\
    &=g^{4k+1}\bigl(a_1(g^d+1)+a_0 \bigr) &\\
	&=g^{4k+1}g^{4i+2} 
	\mbox{(by substituting the values of $a_0$ and $a_1$)}&\\
    &=g^{4(k+i)+3}
\end{flalign*}
\noindent Case(iv):  Evaluating at  $g^j$  for $j\equiv 3\pmod 4$.\\ 
\begin{flalign*}
   f(g^{4(k+i)+3})&= a_2(g^{4(k+i)+3})^{3d+1} +a_1(g^{4(k+i)+3})^{d+1}+a_0 g^{4(k+i)+3} &\\
    &=g^{4(k+i)+3}\bigl(a_2(g^{4(k+i)+3})^{3d}+a_1(g^{4(k+i)+3})^d+a_0 \bigr)& \\
    &=g^{4(k+i)+3}\bigl(a_2g^d-a_1g^d+a_0 \bigr) &\\
    &=g^{4(k+i)+3}(a_1(g^d-1)+a_0) &\\
    &=g^{4(k+i)+3}g^{-4i-2}
	\mbox{  (by substituting the values of }  a_0, a_1) &\\
    &=g^{4k+1}
\end{flalign*}
 This completes the proof of $f(x)$ is an fixed point free involution over $\ff_q$.
\end{proof}


Next  we present the ideas for proving  of  Theorem \ref{t2}:
\begin{proof}
We can show this quadrinomial to be  an involution
	by showing the following pairs for each $i=0,1,\ldots,d-1$ are
	swapped.
 \begin{enumerate}
\item  $g^{4k}\leftrightarrow g^{4(k+i)+3}$
\item  $g^{4k+1}\leftrightarrow g^{4(k+i)+2}.$
\end{enumerate}
	The arguments follow the same lines as in the proof of Theorem 
	\ref{t1} and so omitted.
\end{proof}
\section{Involutions with more fixed points }
Now we move on to provide the proof for Theorem  \ref{t3}.
\begin{proof}
Since $a_1=a_3  $ and  $a_0=a_2+1$,
Given $f(x)$ can be  written as
 \[f(x)=a_3x^{3d+1}+a_2x^{2d+1}+a_3x^{d+1}+(a_2+1)x\]
	\[\hphantom{f()}=a_3\bigl(x^{3d+1}+ x^{d+1}\bigr)+a_2\bigl(x^{2d+1}+x \bigr)+x \]

Enough to verify the condition $f(f(x))=x$ for every $x\in\ff_q$.

	\noindent Case (i):  When $b\in\ff_q^*$ is not a square. 
	In this case we show it is a fixed point. As $d=(q-1)/4$, we have
$b^{2d }= -1$. So 
\begin{align*}
f(b)&=a_3\bigl( b^{3d+1}+b^{d+1} \bigr) +  a_2\bigl(b^{2d+1}+b\bigr)+b &\\
&= a_3b^{d+1}\bigl( b^{2d}+1  \bigr)+a_2b\bigl( b^{2d}+1 \bigr)+b &\\
&=b 
\end{align*}
\noindent Case (ii):  When $b\in\ff_q^*$ is a square.

	In this case we have $b^{2d}=1$.
We proceed to evaluate $f$ at $b$:
\begin{align*}
f(b)&= a_3b^{d+1}\bigl( b^{2d}+1  \bigr)+a_2b\bigl( b^{2d}+1 \bigr)+b &\\
    &=2a_3b^{d+1}+2a_2b+b   &\\
    &=b\bigl( 2a_3b^d+2a_2+1  \bigr) &\\
\end{align*}
Using a simplified notation $\alpha= g^{4i+2}$ and  
substituting the value of $a_3$ and $a_2$ in terms of $\alpha$
	and  we get
\begin{equation*}
f(b) =
\left\{ \begin{array}{rcl}
b\alpha & \mbox{if} & b^d=1  \\ 
b\alpha^{-1} & \mbox{if} & b^d=-1 \\
\end{array}\right.
\end{equation*}
 Composition of $f( f(b) )$ is\\
	\noindent{Subcase (i):} $b^d=1$. Then  $f(b) = b\alpha$.
So we get
\begin{align*}
f(f(b))&=f(b\alpha) &\\
&=b\alpha\{a_3[(b\alpha)^{3d}+(b\alpha)^{d}]+a_2[(b\alpha)^{2d}+1]+1   \} &\\
&=b\alpha \{ a_3[\alpha^{3d}+\alpha^d]+ a_2[\alpha^{2d}+1] +1     \} &\\
&=b\alpha\Bigl( -2a_3+2a_2+1 \Bigr) \mbox{[ since $\alpha^d=g^{2d}=-1$]} &\\
	&=b \mbox{ (by the definition of  $a_3$ and $a_2$) }  &\\ 
\end{align*}
\noindent{Subcase (ii) $b^d=-1$:} 
\begin{align*}
f(f(b))&=f(b\alpha^{-1}) &\\
&=b\alpha^{-1} \{a_3[(b\alpha^{-1})^{3d}+(b\alpha^{-1})^{d}]+a_2[(b\alpha^{-1})^{2d}+1]+1   \} &\\
&=b\alpha^{-1} \{ a_3[b^{3d}\alpha^{-3d}+b^{d}\alpha^{-d}]+ a_2[b^{2d}\alpha^{-2d}+1] +1     \} &\\
&=b\alpha^{-1} \{ a_3[-\alpha^{-3d}-\alpha^{-d}]+ a_2[\alpha^{-2d}+1] +1     \} &\\
&=b\alpha^{-1}\Bigl( 2a_3+2a_2+1 \Bigr) \mbox{[ since $\alpha^{-d}=g^{-2d}=-1$]} &\\
&=b
\end{align*}
 This  completes the proof of $f(f(b))=b$ for every $b\in\ff_q$.
Hence $f(x)$ is an involution  over $\ff_q$. 
\end{proof}

\noindent
\textbf{Remark:}\quad The above quadrinomials in Theorem [\ref{t3}], regarded  
as (involutary) permutations have cycle type \[ 2^{\frac{(q-1)}{4}} +1^{\frac{q+1}{2}} \].
\section{Involutary polynomials with more terms\\  but restricted coefficents}
After involutions from polynomials of 3 and 4 terms we move on to construct 
more of them but from a  $2d$-term polynomial, where $d$ is a divisor of $q-1$.
In this section we consider general $q$ (including characteristic 2)
and write  $q=1\pmod d$ (previously we had $d=4$).
Say,   $q-1=md$.

Now, consider the  set \[\mu_m{:=}\{a\in\ff_{q}: a^m=1\}\] 

\begin{theorem}\label{t5}
 The  polynomial\[h_1(x)=m \sum_{i=0}^{d-1}\Bigl( x^{im+1} - x^{(d-i)m-1} \Bigr)+x\] 
is an involution over $\ff_q$.
\end{theorem}
\begin{proof}
 To prove $h_1(x)$ is an  involution we need to discuss following cases.
Let  $b\in\ff_q$ a nonzero element. Evaluating at $b$ we obtain
\[h_1(b)=m(b+b^{m+1}+\cdots+b^{(d-1)m+1})-m(b^{dm-1}+x^{(d-1)m-1}+\cdots +b^{m-1})+b\]
This can be rewritten as 
	\[h_1(b)=(mb-mb^{-1})(1+b^m+\cdots+b^{(d-1)m}) +b\]
\noindent Case (i),  $b\in \mu_m$:
\begin{align*}
h_1(b)&=mbd-mb^{-1}d+b &\\
	&=-b+b^{-1}+b \ \mbox{( since $md=q-1$ which is same as $-1$)]} &\\
      &=b^{-1}
\end{align*}
\noindent Case (ii),  $b\notin \mu_m$: Summing the finite geometric series,
we get
	\[ h_1(b)=(mb-mb^{-1})(0)+b =b \]
This completes the proof that $h_1(x)$ is an involution.
This polynomial  represents the  function
as described below:
	\begin{equation} 
		h_1(x) =
\left\{ \begin{array}{lcl}
	0  & \mbox{ if  $x=0$ }  \\
    x& \mbox{ if $x\notin \mu_m$} \\
	x^{-1}  & \mbox{ if $x\in \mu_m$} \\	
\end{array}\right.
	\end{equation}
\end{proof}
\begin{theorem} \label{t6}
 The  polynomial\[h_2(x)=x^{dm-1}+m \sum_{i=0}^{d-1}\Bigl( x^{(d-i)m-1}-x^{im+1} \Bigr)\] 
is an involution over $\ff_q$.
\end{theorem}
\begin{proof}
Proof is similar to previous theorem. The polynomial $h_2(x)$ represents
the  function admitting the description below:
\begin{equation} 
		h_2(x) =
\left\{ \begin{array}{lcl}
	0  & \mbox{ if  $x=0$ }  \\
    x& \mbox{ if $x \in \mu_m$} \\
	x^{-1}  & \mbox{ if $x\notin \mu_m$} \\	
\end{array}\right.
	\end{equation}
\end{proof}
\noindent
\textbf{Remark:} Interesting to note that the two
polynomials $h_1(x)$ and $h_2(x)$ 
mentioned in theorems \ref{t5} and \ref{t6}
have a dual relationship in two senses:
(i) Comparing (1) and (2) shows that while
one is an inversion on the subset $\mu_m$  (and identity on the 
complement of it)
the other has  the same behaviour but, with the roles of $\mu_m$ and
the complement reversed.
(ii) one is obtained from the other by reversing the
cefficients. Hence in a suitable algebraic extension, the
nonzero roots of one of the polynomials are
the reciprocals of the nonzero roots of the other. 
\begin{exmple} 
Take a Sophie Germain prime $q=2p+1$ where $p$ is a prime number. While considering elements of order $m=\frac{q-1}{2}$ in theorems \ref{t5}, and \ref{t6}
	we get the following pair of `dual' quadrinomials in $\ff_q[x]$:
	\[ (m+1)x^{2m-1}+(m+1)x^{m+1}+mx^{m-1}+(m+1)x,  \]
\[(m+1)x^{2m-1}+mx^{m+1}+(m+1)x^{m-1}+(m+1)x \] 
Here $m+1=-m=\frac{q+1}{2}$.

For the specific case of  $q=23=2\times 11+1$, we get 
	quadrinomials in $\ff_{23}[x]$:
\[12x^{21} + 12x^{12} + 11x^{10} + 12x, \qquad
	12x^{21} + 11x^{22} + 12x^{10} + 12x  \]
\end{exmple}
\begin{exmple} 
Consider a field of order $q=2^{2n}$ for $n>1$. Obviously $3|q-1$, define $m=\frac{q-1}{3}$.  Theorem  
\ref{t5} in this case gives us an  involution 
	from a 5-term  polynomial
	 \[ x^{2m+1}+ x^{2m-1}+ x^{m+1}+x^{m-1}+x \]
While theorem \ref{t6} gives a different polynomial 
\[   x^{3m-1}+x^{2m+1}+ x^{2m-1}+ x^{m+1}+x^{m-1}\] 
which is also an involution.
\end{exmple}

\subsection*{Quadrinomials with restricted coefficients}
Now we go back to odd characteristic.
Let $m,k,n$  be positive integers such that $nm=q-1$ with $m$ is an odd, $(m,n)=1$ and $k=\frac{n}{2}$. 
 We are ready to present a quadrinomial
where any two  nonzero coefficents are  either equal or additive inverses.

\begin{theorem} \label{t7}
Let $t=2m^{-1} \pmod k$ be a  least positive integer  modulo $k$. Consider a polynomial\[f_1(x)=\frac{1}{2}(x^{(k+t)m-1}-x^{km+1}+x^{tm-1}+x   )\] is an involution on $\ff_q$.
\end{theorem}

\begin{proof}
 With $g$  as a generator of  the multiplicative group $\ff_q^*$,
nonzero elements can be written as $g^{ni+mj},$ for $0\leq i\leq m-1$ and $0\leq j\leq n-1$.

So  $f_1(x)$ can be rewritten as \[f_1(x)=\frac{1}{2}(x^{(k+t)m-1}+x^{km+1}+x^{tm-1}+x )-x^{km+1}\]
We will prove the theorem  by evaluating $f_1(x)$ at  all nonzero elements.
\begin{flalign*}
f(g^{ni+mj})&=\frac{1}{2}\Bigl( (g^{ni+mj})^{(k+t)m-1}+ (g^{ni+mj})^{km+1} +(g^{ni+mj})^{tm-1} + g^{ni+mj} \Bigr)-(g^{ni+mj})^{km+1}) &\\
&=\frac{1}{2}\Bigl( (g^{(k+t)m-1})^{ni+mj}+ (g^{km+1})^{ni+mj} +(g^{tm-1})^{ni+mj} + g^{ni+mj} \Bigr)-(g^{km+1})^{ni+mj}) &\\
&=\frac{1}{2}\Bigl( (-1)^{ni+mj}(g^{tm-1})^{ni+mj}+(-1)^{ni+mj}g^{ni+mj}+(g^{tm-1})^{ni+mj}+g^{ni+mj}\Bigr)-(-1)^{ni+mj}g^{ni+mj} &\\
	&=\frac{1}{2}\Bigl( (-1)^{mj}(g^{tm-1})^{ni+mj}+(-1)^{mj}g^{ni+mj}+(g^{tm-1})^{ni+mj} + g^{ni+mj} \Bigr)-(-1)^{mj}g^{ni+mj} &\\  &=\frac{1}{2}\Bigl( [1+(-1)^{mj}][g^{ni+mj}+(g^{tm-1})^{ni+mj}] \Bigl)-(-1)^{mj}g^{ni+mj} \qquad \ldots(\mathrm{Eqn1})\\ 
\end{flalign*}
\noindent Case 1: When $i=0$  \\
 Now $f(x)$ at $g^{mj}$
\begin{align*}
f(g^{mj}) &=\frac{1}{2}\Bigl( [1+(-1)^{mj}][g^{mj}+(g^{tm-1})^{mj}] \Bigl)-(-1)^{mj}g^{mj} &\\
\end{align*}
\noindent Subcase 1(a),  $j$ is  odd:
\begin{align*}
f(g^{mj})&= g^{mj}
\end{align*}
\noindent Subcase 1(b),  $j$ is  even:
\begin{align*}
f(g^{mj})&=\frac{1}{2}\Bigl( 2g^{mj}+2(g^{tm-1})^{mj} \Bigr)- g^{mj} &\\
         &= (g^{tm-1})^{mj} &\\
         &= g^{mj} \mbox{[by substituting the value of $t$ ]}
\end{align*}

	\noindent Case 2a) : $i>0$, and  $j$ odd:\\
	In this case  (Eqn1)  above simplifies to   
    $f(g^{ni+mj})= g^{ni+mj} $ making it a fixed point.

	\noindent Case 2b): $i>0$,  $j$ even:
\begin{align*}
f(g^{ni+mj})&=\frac{1}{2}\Bigl( 2g^{ni+mj}+2(g^{tm-1})^{ni+mj} \Bigr)- g^{ni+mj} &\\
         &= (g^{tm-1})^{ni+mj} &\\
         &= (g^{tm})^{ni+mj}(g^{-ni-mj} &\\
         &= (g^{tm})^{mj}(g^{-ni-mj}) \mbox{[ since $g^{nm}=1$]  }&\\
         &= g^{-ni+mj} \mbox{[by substituting the value of $t$ ]}
\end{align*}
This completes the proof.
Moreover, this shows the above polynomial
	admits a simpler alternative description as below:
\begin{equation} 
		f(g ^{ni+mj}) =
\left\{ \begin{array}{lcl}

	g ^{mj}  & \mbox{ if $i=0$}  \\
	g ^{ni+mj}  & \mbox{ if $i >0,j$ is an odd } \\
	g ^{-ni+mj}  & \mbox{ if $i >0,j$ is an even } \\
\end{array}\right.
	\end{equation}
\end{proof}
\begin{theorem} \label{t8}

Under the notation described above, the following polynomials
\begin{equation*}
f(x) =
\left\{ \begin{array}{lcl}
	\frac{1}{2}(x^{(2k-2)m-1}+x^{km+1}-x^{(k-2)m-1}+x) & \mbox{\rm if}  & m\equiv -1\pmod k \\[2pt]
	\frac{1}{2}(-x^{(k+2)m-1}+x^{km+1}-x^{2m-1}+x) & \mbox{\rm if} & m\equiv 1\pmod k    \\[2pt]
	\frac{1}{2}(x^{(k+1)m-1}+x^{km+1}-x^{m-1}+x) & \mbox{\rm if} & m\equiv 2\pmod k   \\[2pt]
 \frac{1}{2}(-x^{(2k-1)m-1}+x^{km+1}+x^{(k-1)m-1}+x) & \mbox{\rm if} & m\equiv -2\pmod k   \\ 
\end{array}\right.
\end{equation*}
are involutions over $\ff_q$ .
\end{theorem}

Proof for the first polynomial:
 When  $m\equiv -1\pmod k$:  

\begin{proof}
Given \[ \frac{1}{2}(x^{(k-2)m}(x^{km}-1)+x(x^{km}+1))\]
Evaluating the polynomial at all nonzero elements
\begin{flalign*}
f(g^{ni+mj})&=\frac{1}{2}\Bigl((g^{ni+mj})^{(k-2)m-1}((g^{ni+mj})^{km}-1)+(g^{ni+mj})((g^{ni+mj})^{km}+1)\Bigr)&\\
&= \frac{1}{2}\Bigl((g^{ni+mj})^{(k-2)m-1}((-1)^{ni+mj}-1)+(g^{ni+mj})((-1)^{ni+mj}+1)\Bigr)  &\\
&= \frac{1}{2}\Bigl((g^{ni+mj})^{(k-2)m-1}[(-1)^{mj}-1]+(g^{ni+mj})[(-1)^{mj}+1]\Bigr)  &\\
&= \frac{1}{2}\Bigl( (-1)^{mj}(g^{-2m-1})^{ni+mj} [(-1)^{mj}-1] +g^{ni+mj}[(-1)^{mj}+1]\Bigr)  &\\
	&= \frac{1}{2}\Bigl( (g^{-2m-1})^{ni+mj} [1-(-1)^{mj}] +g^{ni+mj}[(-1)^{mj}+1]\Bigr) \ldots (Eqn2) &\\
\end{flalign*}
\noindent Case 1,  $i=0$.
	\[f(g^{mj})=\frac{1}{2} \Bigl((g^{-2m-1})^{mj} [1-(-1)^{mj}]+g^{mj}[(-1)^{mj}+1]  \Bigr) \]
In this equation considering  $j$ odd and even separately,
we can easily see that $g^{mj}$ is a fixed point
	
\noindent Case (2a), $i\ge1$ and  $j$ odd:  \\
	In this case   (Eqn2) above  simplifies to
\begin{align*}
    f(g^{ni+mj})&= (g^{-2m-1})^{ni+mj} &\\ 
             &= (g^{-2m-1})^{ni} (g^{-2m-1})^{mj} &\\
             &=g^{-ni}g^{mj} &\\
             &=g^{-ni+mj}
\end{align*}
	\noindent Case (2b), $i\ge1$ and  $j$ even:
	Again (Eqn2) above with $j$ even allows us to see
	  \[f(g^{ni+mj})=g^{ni+mj}\]
is a fixed-point. This completes the proof.
Summarising we get an alternative description of the polynomial
	as below:
\begin{equation*} 
		f(g ^{ni+mj}) =
\left\{ \begin{array}{lcl}

	g ^{mj}  & \mbox{ if $i=0$}  \\
	g ^{-ni+mj}  & \mbox{ if $i >0,\ j$ is odd } \\
	 g ^{ni+mj} & \mbox{ if $i >0,\ j$ is even } \\
\end{array}\right.
	\end{equation*}
\end{proof}

  The  proof  for the other three polynomials in $\ref{t8}$  uses 
  similar arguments and so omitted.
These quadrinomials in Theorems [\ref{t7}, \ref{t8}], as permutations 
have cycle type 
$ 2^{\frac{n(m-1)}{4}} +1^{\frac{n(m+1)+2}{2}}$.

\section*{Appendix}
Here we list the permutation trinomials obtained by applying 
our theorems. These are all involutions over the field $\ff_{41}$ some of them with  0 as the only fixed point and others having 21 fixed-points.

\vspace{6pt}
Involutions with single fixed point, from Theorem \ref{t1}:  

\noindent  $ 26x^{31} + 29x^{11} + 22x, \quad
 3x^{31} + 27x^{11} + 9x, \quad
36x^{31} + 37x^{11}, \quad
3x^{31} + 27x^{11} + 32x,$

\noindent $26x^{31} + 29x^{11} + 19x,\quad
15x^{31} + 12x^{11} + 19x,  \quad
38x^{31} + 14x^{11} + 32x,\quad
5x^{31} + 4x^{11},$ 

\noindent
$38x^{31} + 14x^{11} + 9x,\quad 
15x^{31} + 12x^{11} + 22x. $

\vspace{6pt}
Involutions with single fixed point, from Theorem \ref{t2}:  

\noindent
	$ 11 x^{31} + 30 x^{21} + 32 x^{11} + 20 x,\quad
 6 x^{31} + 16 x^{21} + 10 x^{11} + 38 x,\quad
 	31 x^{31} + 38 x^{11},  $

\noindent
 $6 x^{31} + 25 x^{21} + 10 x^{11} + 3 x,\quad	
11 x^{31} + 11 x^{21} + 32 x^{11} + 21 x,\quad	
30 x^{31} + 11 x^{21} + 9 x^{11} + 21 x,   $

\noindent
$ 35 x^{31} + 25 x^{21} + 31 x^{11} + 3x,\quad 	
 10 x^{31} + 3 x^{11},\quad
 35 x^{31} + 16 x^{21} + 31 x^{11} + 38 x,\quad 
 30 x^{31} + 30 x^{21} + 9 x^{11} + 20 x. $

\vspace{6pt}
Involutions with more fixed points, from Theorem \ref{t3}(a):  

\noindent
 $7 x^{31} + 31 x^{21} + 7 x^{11} + 32 x,\quad 
 15 x^{31} + 4 x^{21} + 15 x^{11} + 5 x,\quad
16 x^{31} + 20 x^{21} + 16 x^{11} + 21 x,$

\noindent
$15 x^{31} + 36 x^{21} + 15 x^{11} + 37 x,\quad
 7 x^{31} + 9 x^{21} + 7 x^{11} + 10 x,	\quad
 34 x^{31} + 9 x^{21} + 34 x^{11} + 10 x,$

\noindent
  $26 x^{31} + 36 x^{21} + 26 x^{11} + 37 x,\quad
25 x^{31} + 20 x^{21} + 25 x^{11} + 21 x, \quad
 26 x^{31} + 4 x^{21} + 26 x^{11} + 5 x,$

\noindent
$34 x^{31} + 31 x^{21} + 34 x^{11} + 32 x.$

\vspace{6pt}
Involutions with more fixed points, from Theorem \ref{t3}(b):  

\noindent
$19 x^{31} + 10 x^{21} + 22 x^{11} + 32 x,\quad
29 x^{31} + 37 x^{21} + 12 x^{11} + 5 x,\quad
 20 x^{31} + 21 x^{21} + 21 x^{11} + 21 x,$ 

\noindent
$ 29 x^{31} + 5 x^{21} + 12 x^{11} + 37 x,\quad
19 x^{31} + 32 x^{21} + 22 x^{11} + 10 x,	\quad
22 x^{31} + 32 x^{21} + 19 x^{11} + 10 x,\quad
 $

\noindent
 $12 x^{31} + 5 x^{21} + 29 x^{11} + 37 x,\quad
 21 x^{31} + 21 x^{21} + 20 x^{11} + 21 x,\quad
   12 x^{31} + 37 x^{21} + 29 x^{11} + 5 x,$

\noindent
 $22 x^{31} + 10 x^{21} + 19 x^{11} + 32 x.
$

\end{document}